\documentclass{article}

\usepackage{hyperref}
\usepackage{amsmath}
\usepackage{amsthm}
\usepackage{amssymb}
\usepackage{xcolor}
\usepackage{ulem}

\usepackage{graphicx}
\graphicspath{ {./images/} }

\newtheorem{theorem}{Theorem}
\newtheorem{prop}{Proposition}
\newtheorem{definition}{Definition}
\newtheorem{example}{Example}
\newtheorem{lemma}{Lemma}

\newtheorem{remark}{Remark}
\newtheorem{conjecture}{Conjecture}

\newcommand{\FF}{\mathbb{F}}

\begin{document}

\title{A result on the $c_2$ invariant for powers of primes.}

\author{Maria S. Esipova, Karen Yeats\footnote{KY is supported by an NSERC Discovery Grant and the Canada Research Chairs program.  MSE was supported by an NSERC USRA. KY thanks Oliver Schnetz for many discussions on the $c_2$ invariant over the years.}}

\maketitle              

\begin{abstract}
The $c_2$ invariant is an arithmetic graph invariant related to quantum field theory.  We give a relation modulo $p$ between the $c_2$ invariant at $p$ and the $c_2$ invariant at $p^s$ by proving a relation modulo $p$ between certain coefficients of powers of products of particularly nice polynomials.  The relation at the level of the $c_2$ invariant provides evidence for a conjecture of Schnetz. 
\end{abstract}


\section{Introduction}

The $c_2$ invariant is an arithmetic graph invariant defined by Oliver Schnetz in \cite{SFq}. It is defined based on counting points over finite fields on the Kirchhoff polynomial of a graph, as we will see below.  

Quantum field theory provides motivation for studying the $c_2$ invariant. Considering some graph as a scalar Feynman diagram, its Feynman integral can be written using the Kirchhoff polynomial and the geometric content of the integral is substantially controlled by the geometry of the variety defined by the vanishing of the Kirchhoff polynomial. In particular, a residue of the Feynman integral known as the Feynman period is an integral over the inverse square of the Kirchhoff polynomial. The geometry of the Kirchhoff variety determines the kinds of numbers the Feynman integral can involve, see \cite{bek, BKphi4, Sphi4, Snumfun}. Both the $c_2$ invariant and the Feynman period are thus probing the geometry of the Kirchhoff variety and so are intimately related.

The $c_2$ invariant is also interesting from a pure mathematics perspective because of the interesting sequences of numbers which appear, including coefficient sequences of $q$-expansions of modular forms, see \cite{BrS, BrS3, Lmod, further}. 

Often the $c_2$ invariant is only considered for primes even though it is defined for all prime powers.  In particular, the identification with modular forms involves only the primes. Schnetz conjectured (see Conjecture~\ref{oliver} below) that knowing the values on all primes suffices to determine the $c_2$ invariant on all prime powers. We prove a weaker result in that direction in our main result, Theorem~\ref{thm1}.

\subsection{Background and set up}

Consider a connected 4-regular simple graph $K$, and let $G= K-v$ for some vertex $v$ of $K$.  We say $G$ is a \textbf{decompletion} of $K$ and $K$ is the \textbf{completion} of $G$.
$G$ can be viewed as a Feynman graph in $\phi^4$ theory with 4 external edges. The external edges are given by the edges incident to $v$, but no longer joined at $v$.  Note that while multiple edges and self-loops can appear in quantum field theory, they always result in subdivergences and since we will soon be restricting to the primitive, that is, subdivergence-free, case, there is no loss in working only with simple graphs.
Associate a parameter $\alpha_e$ to each edge $e$ of $G$.

\begin{definition}
The \textbf{Kirchhoff polynomial} of $G$ is 
\[
\Psi_G = \sum_T \prod_{e \notin T} \alpha_e
\]
where the sum runs over all spanning trees of $G$.
\end{definition}

We further define the \textbf{Feynman period} of $G$ from the Kirchhoff polynomial by
\[
\int_{\alpha_i \geq 0} \dfrac{\Omega}{\Psi_G^2} 
\]
when it converges, where $\Omega = \sum_{i=1}^{|E(G)|} (-1)^{i}\alpha_i d\alpha_1 \ldots \widehat{d\alpha_i} \ldots d\alpha_{|E(G)|}$.  This is the Feynman period in projective parametric form.  For other forms see \cite{Sphi4}.

The Feynman period of $G$ is known to converge in the case that $K$ (the completion of $G$) is internally 6-edge connected, that is, in the case where removing any 5 edges of $K$ results in at most one connected component that is not an isolated vertex. In the case that the completion of $G$ is internally 6-edge connected we say $G$ is \textbf{primitive divergent} \cite{Sphi4}.
Feynman periods are hard to evaluate but give number-theoretically interesting results, see for instance \cite{Sphi4}.  The geometry of the integral is controlled by the geometry of the denominator and it is therefore natural to study rational points on the variety $\Psi_G = 0$ over various finite fields. With this motivation, we define the combinatorial graph invariant, the \textbf{$c_2$ invariant}.

For a prime power $q$, let $\FF_q$ be the finite field with $q$ elements. For polynomial $F \in \mathbb{Z}[x_1, \ldots x_N]$, let $[F]_q$ be the point counting function
\[
    [F]_q : q \rightarrow |\{x \in \FF_q : F(x)=0\}|.
\] 

\begin{definition}
Let $q = p^s$ for $p$ prime and $s \geq 1$. For $G$ with $|V(G)| \geq 3$, the \textbf{$c_2$ invariant} of $G$ at \textbf{$q$} is
\[
c_2^{(q)}(G) = \dfrac{[\Psi_G]_q}{q^2} \mod q.
\]
\label{def:c2}
\end{definition} 
The $c_2$ invariant is well-defined given $G$ is connected and has at least 3 vertices \cite{SFq}. 
The fact that we define the $c_2$ invariant at $q$ as a residue modulo $q$ rather than as a nonnegative integer is primarily to be consistent with other work on the $c_2$ invariant.  Little would change defining it without the modulo $q$ --- what are by our definition $c_2$ identities for the residue definition would become congruences modulo $q$.  While there has been some consideration of the nonnegative integer version of the $c_2$ invariant, to our knowledge the exact integers do not seem to be particularly meaningful or structured, while the residues tell us something about the underlying geometry, giving an additional justification for the definition that we give.

In the 90's, all known Feynman integrals were expressible in terms of multiple zeta values.  Inspired by this, in 1997 Kontsevich informally conjectured that $[\Psi_G]_q$ might always be a polynomial in $q$, in which case the $c_2$ invariant would capture the quadratic coefficient of $[\Psi_G]_q$. Neither the expressibility of Feynman periods in terms of multiple zeta values nor the polynomiality of $[\Psi_G]_q$ are true in general, see \cite{BrBe, SFq, BrS, Doreg}. However the $c_2$ invariant remains a useful measure of how far from true they are for a given graph.


In view of such close connections, Schnetz (Remark 2.11 (2) in \cite{SFq}) and Brown and Schnetz (Conjecture 2 in \cite{BrS3}) conjectured that if two primitive divergent graphs have the same period then they have the same $c_2$ invariant for all prime powers $q$. This is supported by all current known evidence. Since the Feynman period is known not to depend on the decompletion vertex $v$ \cite{Sphi4}, this conjecture implies the weaker conjecture of Brown and Schnetz \cite{BrS} --- the $c_2$ invariant does not depend on the decompletion vertex. This conjecture is called the \textbf{completion conjecture} and recent work of Simone Hu proves the $q=2$ case \cite{Hmmath, HYcompletion}.

Most crucially for us, in \cite{Sgeometries} Oliver Schnetz makes the conjecture that $c_2$ invariants only have to be considered for primes,  
\begin{conjecture}[Conjecture 2 in \cite{Sgeometries}]\label{oliver}
Let $G_1$, $G_2$ be two graphs with $c_2^{(p)}(G_1) \equiv c_2^{(p)}(G_2) \mod p$ for all primes $p$, then $c_2^{(q)}(G_1) \equiv c_2^{(q)}(G_2) \mod q$ for all prime powers $q = p^s$.
\end{conjecture}
The mod $q$ in the statement of the conjecture is redundant since the $c_2$ invariant at $q$ is defined as a residue modulo $q$, however it is included to emphasize the contrast of $p$ and $q$ between the conjecture and our result.

Much of the computational work on $c_2$ invariants has only been on primes for reasons of computational simplicity; this conjecture would justify this simplification. Here we will prove the weaker result that if $c_2^{(p)}(G_1) \equiv c_2^{(p)}(G_2) \mod p$ then $c_2^{(q)}(G_1) \equiv c_2^{(q)}(G_2) \mod p$ for all powers $q$ of $p$ --- the conjectured relation holds modulo $p$, rather than the stronger modulo $q = p^s$, see Theorem~\ref{thm1}.
Note that our approach to the problem depends on converting point counting to coefficient extraction, a result which is intrinsically modulo the prime $p$, not $q$, so our methods are not suitable to proving the full conjecture and the result we give is best possible in this direction given the limitations of these methods.  This is discussed in more detail immediately following Lemma~\ref{lma3}. 

\section{Calculating $c_2$ invariants by counting edge partitions}\label{sec edge part}

Following up on the recent successes of one of us with other coauthors \cite{CYgrid, Ycirc, Yscompl, Yprefix}, including the recent progress on the completion conjecture \cite{Hmmath, HYcompletion}, we will approach calculating $c_2$ invariants in a combinatorial way involving counting certain edge partitions.  We need a few definitions leading up to this.

We can express the Kirchhoff polynomial in the form of a determinant. Choose an arbitrary orientation of the edges of $G$ and let $\tilde{E}$ be the $|V(G)| \times |E(G)|$ signed incidence matrix given by
\[
(\tilde{E})_{v,e} = 
\begin{cases}
+1 & \text{if edge $e$ begins at $v$} \\
-1 & \text{if edge $e$ ends at $v$} \\
0 & \text{otherwise}
\end{cases}
\]
$\tilde{E}$ does not have full rank so we let $E$ equal to $\tilde{E}$ with one row removed.
Let $A$ be the $|E(G)| \times |E(G)|$ matrix with the parameters $\alpha_e$ on the diagonal, zeroes elsewhere. Let 
\[
M = \begin{bmatrix}
A & E^T \\
-E & 0
\end{bmatrix}.
\]
By the matrix-tree theorem (see \cite{Brbig} for a proof), 
\[
\Psi_G = \det M.
\]

We are also interested in minors of $M$. Following Brown we make the following definition of \textbf{Dodgson polynomials} \cite{Brbig}.
\begin{definition}
Let $I, J, K$ be subsets of $\{1, \ldots, |E(G)| \}$ with $|I| = |J|$. Let
\[
\Psi^{I,J}_{G,K} = \det M(I,J) \mid_{\alpha_e = 0, \, e \in K}
\]
where $M(I,J)$ is the matrix $M$ with rows indexed by $i \in I$ and columns indexed by $j \in J$ removed.
\end{definition}
Note we suppress $G$ from the notation when clear from context and similarly omit $K = \emptyset$. 
As given, Dodgson polynomials are only well-defined up to sign. This ambiguity was resolved by Schnetz in \cite{Sgeometries}, but the unsigned version will suffice for the present purposes.

To view Dodgson polynomials in a more combinatorial way rather than from determinants we make the following definition:
\begin{definition}
Let $P$ be a partition of a subset of the vertices of $G$. The \textbf{spanning forest polynomial of $G$ with respect to $P$} is 
\[
\Phi^P_G = \sum_F \prod_{e \notin F} \alpha_e
\]
where the sum runs over all spanning forests $F$ of $G$ \textbf{compatible} with $P$ in the sense that each tree of the forest contains all vertices in one part of $P$ and no vertices from another part of $P$. 
\end{definition}
Here by a spanning forest we mean a subgraph of $G$ that includes all vertices of $G$ (that is, it is spanning) and has no cycles (that is, it is a forest).  Note we allow singleton trees in forests and we illustrate a partition of the vertices on a graph by coloring some of the vertices.  Another way to phrase the compatibility condition is that there is a bijection between the parts of the partition and the trees of the forest such that the vertices of a part are contained in the corresponding tree.
When it is unambiguous we will simplify the notation and write for instance $P=ab,c$ instead of $P = \{ a,b\}, \{c \}$.

\begin{figure}
    \centering
    \includegraphics{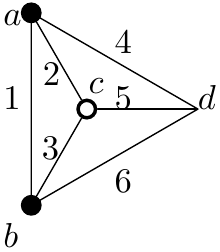}
    \caption{An example of a primitive divergent graph with a partition of a subset of vertices represented by vertex colourings.}
    \label{fig sp}
\end{figure}

\begin{example}
Consider the graph given in Figure~\ref{fig sp} with partition $P = \{a,b\}, \{c\}$ represented by the colouring of three of the four vertices. Then the spanning forest polynomial with respect to $P$ is $\Phi^P = \alpha_2\alpha_3(\alpha_4\alpha_5+\alpha_4\alpha_6 + \alpha_5\alpha_6 + \alpha_1\alpha_5)$.
\end{example}

\begin{remark}\label{rem linear}
Observe that Dodgson polynomials and spanning forest polynomials are linear in each variable.  For the spanning forest polynomials this is immediate from the definition, as it is for the Kirchhoff polynomial itself.  For the Dodgson polynomials it follows directly from the fact that the matrix $M$ contains each variable exactly once so any term in any determinant of any submatrix of $M$ is at most linear in each variable, or alternately can be seen as an immediate corollary of Lemma~\ref{lma1}.
\end{remark}

In order to discuss the $c_2$ invariant from a combinatorial perspective, we use the following lemmas to transform the problem of polynomials and point-counting to one of counting partitions of the edges into spanning trees and forests. \\

The first step is to express Dodgson polynomials as sums of spanning forest polynomials.  To interpret $\Psi^{I,J}_{G,K}$, the key is to look at spanning forests which become trees in both $G \setminus I \, /\, (J \cup K)$ and $G \setminus J \, /\, (I \cup K)$. In other words, subsets of the edges of $G\setminus (I\cup J\cup K)$ that form an acyclic, not necessarily connected subgraph of $G$, such that after removing edges in $I$ (respectively $J$), and contracting edges in $J$ and $K$ (respectively $I$ and $K$) our subset forms a now connected, still acyclic
graph. Conveniently this condition can be captured by some set partitions of the endpoints of the edges in $I\cup J\cup K \setminus (I\cap J)$ leading to the sum of spanning forest polynomials of the next lemma.  Example~\ref{3-inv} gives an illustrative example of how this works out in a useful situation.  In fact, the calculation of Example~\ref{3-inv} is all that we will need from Lemma~\ref{lma1} in the subsequent arguments.

\begin{lemma}[Corollary 7 and Propositions 8 and 12 in \cite{BrY}]\label{lma1}
Let $I, J, K$ be subsets of $\{1, \ldots, |E(G)| \}$ with $|I| = |J|$. Then
\[
\Psi^{I,J}_{G,K} = \sum_P \pm \Phi^P_{G \setminus {I \cup J \cup K}}
\]
where the sum runs over partitions $P$ of the endpoints of edges in $I \cup J \cup K \setminus (I \cap J)$ such that all spanning forests compatible with $P$ become spanning trees in both $G \setminus I \, /\, (J \cup K)$ and $G \setminus J \, /\, (I \cup K)$. 
\end{lemma}

\begin{proof}
In the case $I\cap J = \emptyset$ and $K=\emptyset$
Proposition 12 of \cite{BrY} gives $\Psi^{I,J}_{G,K}$ as a signed sum of spanning forest polynomials defined by certain partitions of the endpoints of $I$ and $J$.  Proposition 8 of \cite{BrY} specifies that the partitions required are precisely those for which a compatible spanning forest becomes a spanning tree in $G\setminus I \,/ \,J$ and in $G\setminus J\,/\,I$.  

Corollary 7 of \cite{BrY} lets us reduce to the case $I\cap J=K=\emptyset$ from the general case by contracting edges of $K$ and deleting edges of $I\cap J$.  Then by the previous paragraph we get the desired result on $G\setminus (I\cap J) \, /\, K$. Reinterpreting the partitions and polynomials on $G$ directly gives the lemma.
\end{proof}

In Lemma~\ref{lma1} and throughout we are using the usual graph theory notions of edge deletion and edge contraction denoted $\backslash$ and $/$ respectively.  In particular, when $S$ is a subset of vertices, $G \setminus S$ refers to $G$ with all edges in $S$ removed. 
However, in some other references, including \cite{BrY}, if edge deletion results in isolated vertices then those isolated vertices are removed.  Here we are not using that convention, but note that the choice of convention regarding isolated vertices does not change the lemma or the results used in its proof. All results from \cite{BrY} can be adapted for either convention since in a spanning forest polynomial an isolated vertex must be a singleton tree in every forest contributing to the forest polynomial, so isolated vertices must appear as singleton parts in the partition and otherwise have no effect on the polynomial.

As with Dodgson polynomials, we suppress the graph subscript from the notation when obvious from context and assume all needed edges have been deleted.

The signs in Lemma~\ref{lma1} are worth some discussion.
Proposition 12 of \cite{BrY} gives an explicit way to calculate the signs of the spanning forest polynomials in the sum. However, if the sum in the lemma contains only a single spanning forest polynomial, then the only sign is the overall sign. Since we are ultimately interested in counting points on products of such polynomials, the overall sign is irrelevant.


The examples we require are the following: 

\begin{figure}
    \centering
    \includegraphics{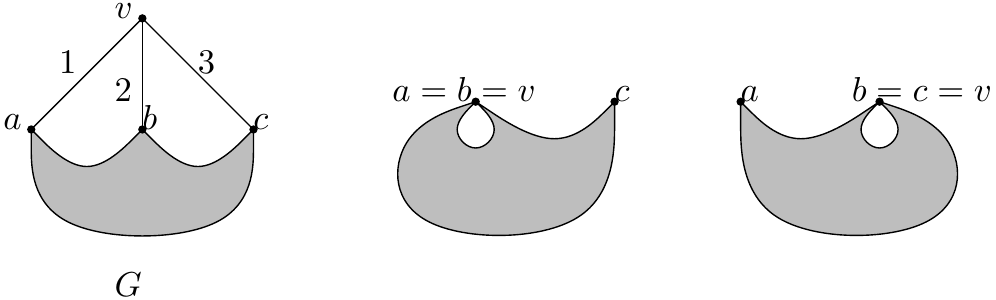}
    \caption{An example primitive divergent graph $G$ with two of its minors as discussed in Example~\ref{3-inv} }\label{fig eg G}
\end{figure}

\begin{example}\label{3-inv}
  Suppose we have a graph $G$ with a 3-valent vertex and with edge and vertex labels as in the first graph in Figure~\ref{fig eg G}. We want to find an expression for $\Psi^{12,23}_G$ and $\Psi^{1,3}_{G,2}$ using Lemma~\ref{lma1}.

By Lemma \ref{lma1}, to find the expression for Dodgson polynomial $\Psi^{12, 23}_G$ in terms of spanning forest polynomials we need to look for sets of edges of $G\backslash 123$ which give a spanning tree in $G\backslash 12/3$ and in $G\backslash 23/1$.  These smaller graphs are each isomorphic to $G-v$ and so we obtain
  \[
  \Psi^{12,23}_G = \pm \Psi_{G-v}. 
  \]
  Note that $v$ is an isolated vertex in $G\backslash 123$ (taking the graph theory convention on edge deletion), so we could also write $\Psi_{G-v}$ as $\Phi^{\{v\}, \{a,b,c\}}_{G\backslash 123}$, which is the form that would come directly from Lemma \ref{lma1}. In what follows we remove isolated vertices from the graph, as well as the singleton part from the partition corresponding to the vertex without comment whenever it is convenient.
  
  Since this application of Lemma~\ref{lma1} involves only one spanning forest polynomial in the sum on the right-hand side, the only possible sign is the overall sign which will not be relevant for our (next) desired result, Lemma~\ref{lma2}.
  
  Similarly, to find the expression for the Dodgson polynomial $\Psi^{1,3}_{G,2}$ in terms of spanning forest polynomials we need to look for sets of edges of $G \backslash 123$ which give a spanning tree in $G\backslash 3/12$ and in $G\backslash 1/23$.  These two smaller graphs are the second and third graphs in Figure~\ref{fig eg G} respectively.  If a set of edges is a tree for both of these graphs then on $G\backslash 123$ we must have $a$ and $b$ in different trees as well as $b$ and $c$ in different trees to avoid a cycle. The vertex $v$ must also be in its own tree. Furthermore, to get a tree in $G\backslash 3/12$ and in $G\backslash 1/23$ the edges viewed in $G\backslash 123$ must give a spanning forest with exactly three trees. There is one partition of $\{v, a, b, c\}$ which satisfies these properties, namely $\{v\},\{b\},\{a,c\}$, such that all forests compatible with this partition give trees in the two minors of Figure~\ref{fig eg G} as required. Therefore
  \[
  \Psi^{1,3}_{G,2} = \pm \Phi^{v, b, ac}_{G\backslash 123} = \pm \Phi^{b, ac}_{G-v}
  \]
  where the last equality holds because the vertex $v$ is isolated in $G\backslash 123$.  Again there is only one spanning forest polynomial in the application of Lemma~\ref{lma1} and so the only sign is the overall sign.
\end{example}

The calculations of Example~\ref{3-inv} will be very useful, so let us encapsulate them into the following lemma.
\begin{lemma}\label{lem 3-inv}
Let $G$ have 3-valent vertex $v$ with incident edges $1$, $2$, and $3$ whose other vertices are $a$, $b$, and $c$ respectively, as in the first graph in Figure~\ref{fig eg G}.  Then
\begin{align*}
    \Psi^{12,23}_G & = \pm \Psi_{G-v} \\
    \Psi^{1,3}_{G,2} & = \pm \Phi^{b, ac}_{G-v}
\end{align*}
Furthermore, since the overall sign does not affect the point count, $[\Psi^{1,3}_{G,2}\Psi^{12,23}_G]_q = [\Phi^{b, ac}_{G-v}\Psi_{G-v}]_q$ for any prime power $q$.
\end{lemma}
\begin{proof}
Calculations of Example~\ref{3-inv}.
\end{proof}

We can use the polynomials of Example~\ref{3-inv} to express the $c_2$ invariant as a direct point count modulo $q$, avoiding division by $q^2$ as in Definition \ref{def:c2}.
\begin{lemma}\label{lma2}
Suppose $G$ is a connected graph such that $2 + |E(G)| \leq 2|V(G)|$ and there exists a 3-valent vertex. Let $1,2,3$ be distinct edge indices of $G$ and $q$ a prime power. Then
\[
c_2^{(q)}(G) = -[\Psi^{1,3}_2 \Psi^{12,23}]_q \mod q
\]
\end{lemma}
\begin{proof}
When edges $1,2,3$ meet at a 3-valent vertex, Lemma 24 of \cite{BrS} (along with an elementary application of Euler's formula relating the number of edges, vertices, and dimension of the cycle space) gives the result at the level of the Grothendieck ring.  Passing to point counts, Corollary 28 of \cite{BrS} tells us that any choice of three edges has the same point count modulo $q$, proving the lemma.  The exact statement of the lemma also appears as an unnumbered comment immediately after the proof of Corollary 28 in \cite{BrS}.
\end{proof}
Note that any connected graph which is the result of decompleting a 4-regular graph satisfies the hypotheses since for such a graph $|E(G)| = 2(|V(G)|-1)$.  In particular, any primitive divergent graph $G$ is a decompletion of a connected 4-regular graph and so satisfies the hypotheses of the lemma.

Lemma~\ref{lma2} is a special
case of a much more general result that says that the denominators obtained by applying Brown’s
denominator reduction algorithm \cite{Brbig} to the Feynman period can each be used to
calculate the $c_2$ invariant, see Theorem 29 and Corollary 28 of \cite{BrS}.

Finally, the next result allows us to express the point-count modulo $p$ as a coefficient.
\begin{lemma}\label{lma3}
Let $q = p^s$ for $p$ prime. Let $F \in \mathbb{Z}[x_1, \ldots x_n]$ be a polynomial of degree $n$ in $n$ variables with integer coefficients. Then, the coefficient of $(x_1 \ldots x_n)^{q-1}$ in $F^{q-1}$ is $(-1)^{n+1}[F]_q$ modulo $p$. 
\end{lemma}
This is a consequence of one of the standard proofs of the Chevalley-Warning theorem.  See \cite{Sgeometries} for a discussion and proof in slightly different notation. Note that this lemma only holds modulo $p$, not modulo $q$. For example, following \cite{Sgeometries}, if $F=2x$ and $q=9$ then we see that $[F]_9 = 1$ but the coefficient of $x^8$ in $F^8$ is 256 which is congruent to 1 modulo 3, but is congruent to 4 modulo 9.  The appendix to \cite{Sgeometries} gives an extension of a related result to all powers of primes, but that result only applies in the case where the degrees force a zero result, so it is not useful for the present purposes.

Since Lemma~\ref{lma3} is needed to obtain our combinatorial approach for calculating $c_2$, our approach cannot tackle the full Conjecture~\ref{oliver}, a result modulo $q$. Rather, our Theorem~\ref{thm1}, a result modulo $p$, is the best obtainable result by the combinatorial approach.



We will now use square brackets $[\, \cdot \, ]$ as the coefficient extraction operator, as is standard in enumerative combinatorics. Specifically, $[m]F$ refers to the coefficient of monomial $m$ in the polynomial $F$. This should not cause any confusion with the point counting function which we largely no longer need in view of Lemma~\ref{lma3}.

For any three edges $1$, $2$, and $3$ of $G$, let us observe a few facts about $\Psi^{1,3}_2$ and $\Psi^{12, 23}$.  As noted in Remark~\ref{rem linear}, they are both linear in all of their variables.  Furthermore, the variables $\alpha_1$, $\alpha_2$, and $\alpha_3$ do not appear in either of them.  This follows from the definition of the Dodgson polynomials since for $\Psi^{1,3}_2$ and $\Psi^{12, 23}$ and for each $1\leq i \leq 3$, either the row or the column (or both) containing $\alpha_i$ is removed, or $\alpha_i$ is explicitly set to $0$.  Finally, since $G$ is primitive divergent, $\Psi^{1,3}_2$ has degree $|E(G)|/2 -1$ and  $\Psi^{12,23}$ has degree $|E(G)|/2-2$ as is observed in the proof of Corollary 28 of \cite{BrS} and hence the degree of $\Psi^{1,3}_2 \Psi^{12,23}$ is $|E(G)| - 3$.  This can also be seen from Lemma~\ref{lma1} using the fact that contracting an edge decreases the size of any spanning tree by one and deleting an edge does not affect the size of a spanning tree.


Combining Lemma~\ref{lma2} and Lemma~\ref{lma3}, with $|E(G)| = N$ and $n=N-3$,
\[
c_2^{(q)}(G) = -(-1)^{(|E(G)|-3)+1}[(\alpha_4 \ldots \alpha_N)^{q-1}](\Psi^{1,3}_2 \Psi^{12,23})^{q-1} \mod p 
\]
Using Lemma~\ref{lem 3-inv} when edges $1$,$2$, and $3$ are configured as in Figure~\ref{fig eg G} we get, 
\begin{align}
    c_2^{(q)}(G) = -[(\alpha_4 \ldots \alpha_N)^{q-1}](\Psi_H \Phi_H^{b,ac})^{q-1} \mod p \label{eq:c2q_finalform}
\end{align}
where $H=G-v$.  Note that this coefficient modulo $p$ corresponds to the number of ways (modulo $p$) to partition $q-1$ copies of the variables $\alpha_4, \ldots, \alpha_N$ into $q-1$ monomials in $\Psi_H$ and $q-1$ monomials in $\Phi_H^{b,ac}$ such that when the $2(q-1)$ monomials are multiplied together we get $(\alpha_4 \ldots \alpha_N)^{q-1}$; this is due to the facts that $\Psi_H$ and $\Phi_H^{b,ac}$ are linear in each variable (see Remark~\ref{rem linear}) and have all coefficients in $\{0,1\}$. Furthermore, each variable $\alpha_e$ corresponds to an edge, each term of $\Psi_H$ corresponds to a spanning tree, and each term of $\Phi_H^{b,ac}$ corresponds to a forest compatible with $P$, where the edges of the tree or forest are given by the complement of the corresponding monomial. Thus, by taking the complement of each monomial, the number of ways to partition into monomials as described above is the same as to the number of ways to partition $q-1$ copies of the edges into $q-1$ spanning trees and $q-1$ spanning forests compatible with the partition $b, ac$. 

Thus, the $c_2$ calculation at $q$ modulo $p$ simplifies to
\begin{align*}
c_2^{(q)}(G) = -(&\text{number of partitions of $q-1$ copies of the edges of $H$} \\
&\text{into $q-1$ spanning trees of $H$ and $q-1$ spanning} \\
&\text{forests of $H$ compatible with the partition $b, ac$} ) \mod p
\end{align*}

The case for $p = 2$ is particularly nice, we only need to count the parity of the number of bipartitions of the edges.
Similarly, for powers of two, we still get the parity condition, but now have $2(q-1)$ parts. The complexity of the calculation grows with higher primes and subsequently prime powers.

\begin{remark}\label{rem complementary}
By taking the complement of each set of edges or of each monomial, we can move between interpreting the $c_2$ invariant as counting partitions of $q-1$ copies of the variables into monomials in $\Psi_H$ and $\Phi^{b,ac}_H$ and counting partitions of $q-1$ copies of the edges into spanning trees and spanning forests.

Being combinatorially minded, we find it easier to think in terms of partitioning the edges into trees and forests, and this perspective was very useful in previous work including \cite{Yscompl, HYcompletion} where edge swapping arguments on the trees and forests were used to prove that certain $c_2$ invariants agreed.

A more algebraically minded reader might prefer to think in terms of the monomials directly.  Since our main result can be phrased in terms of suitable polynomials which do not necessarily come from graphs, see Proposition~\ref{prop main prop}, the proofs of the next section will be phrased in terms of partitioning variables, but the reader is encouraged, as we do, to think complementarily of the edge partitioning.

\end{remark}

\section{Main result}

We are now ready to give our main result.

\begin{theorem}\label{thm1}
Let $q = p^s$ for $s \geq 1$, $p$ prime. Let $G$ be primitive divergent. Then, $c^{(q)}_2(G) \equiv (-1)^{s+1} \big(c^{(p)}_2(G)\big)^s \bmod p$.\\
\end{theorem}

To prove the main result, we will first extract the key statement at the level of polynomials, Proposition~\ref{prop main prop}, which we will prove with the help of two lemmas.  Then we will give the proof of Theorem~\ref{thm1} by using the set up of the previous sections to relate the $c_2$ calculations to the result of Proposition~\ref{prop main prop}.

\begin{prop}\label{prop main prop}
Let $q = p^s$ for $s \geq 1$, $p$ prime. Let $P$ and $Q$ be polynomials $P,Q\in \mathbb{Z}[x_1, \ldots, x_N]$ which are each linear in each variable and which have every non-zero coefficient equal to $1$. Then
\[
    [(x_1\cdots x_N)^{q-1}](PQ)^{q-1} = ([(x_1\cdots x_N)^{p-1}](PQ)^{p-1})^s \mod p
\]
\end{prop}

To prove Proposition~\ref{prop main prop} it is convenient to first observe two combinatorial lemmas.  

Both lemmas will be useful since, using the factorization $p^s - 1 = (p-1) (p^{s-1} + \ldots + p + 1)$, we can write $P^{q-1}$ as the product of factors, each with an exponent that is a power of $p$
\begin{equation}\label{eq factor P}
P^{q-1} = \big((P)^{p^{s-1}} \ldots P^{p} P\big)^{p-1}. 
\end{equation}
(and similarly for $Q$).  Consequently, it is valuable to consider the assignments of variables on parts of the form $(P)^{p_i}$.  These are particularly well behaved modulo $p$ because the exponent is a power of $p$.  

The first lemma shows that modulo $p$, we can reduce to considering assignments that assign the same variables to each copy of $P$ among a block of the form $(P)^{p^i}$.

\begin{lemma}\label{lemma prop lem 1}
Let $p$, $P$ and $Q$ be as in Proposition~\ref{prop main prop}.  Let $m$ be a monomial appearing in $(P)^{p^i}$ (respectively $(Q)^{p^i}$).  If any variable appears in $m$ with power strictly between $0$ and $p^i$ then the coefficient of $m$ in $(P)^{p^i}$ (respectively $Q^{(p)^i}$) is $ 0 \bmod p$.  If all variables appear in $m$ with power $0$ or $p^i$ then the coefficient of $m$ is $1 \bmod \, p$.
\end{lemma}

\begin{proof}
Without loss of generality say $m$ is a monomial in $(P)^{p^i}$.

Since $P$ has all coefficients in $\{0,1\}$, the coefficient of $m$ in $(P)^{p^i}$ is the number of ways of partitioning the variables of $m$ into $p^i$ ordered parts each of which is a monomial of $P$.  

Since $P$ is linear in each variable, there are at most $p^i$ copies of any variable in $m$.
Suppose that $k$ copies of variable $x_j$ appear in $m$ with $0 < k < p^i$ and with $j$ minimal (i.e. for all $1 \leq \ell < j$, either $0$ or $p^i$ copies of $x_\ell$ appear in $m$).
Let us consider building partitions of $m$ one variable at a time.  For the first $j-1$ variables either $0$ or $p^i$ copies of the variable appear in $m$.  Since $P$ is linear in each variable, and there are $p^i$ parts, this implies that each variable either appears $0$ times in each part or exactly once in each part.  With the variables $x_1, \ldots, x_{j-1}$ now partitioned, consider how the $k$ copies of $x_j$ can be distributed among the parts. Since $P$ is linear in $x_j$ and since the parts are identical so far, there are ${p^i \choose k}$ possible ways to distribute the copies of $x_j$ and each choice differs from the others only by a permutation of the parts so far. Therefore, the number of ways to distribute the remaining variables is the same for each way of distributing the $x_j$. It follows that ${p^i \choose k}$ divides the total number of ways to partition the variables into $p^i$ monomials of $P$.


By Lucas' theorem, ${p^i \choose k} \equiv 0 \bmod p$ since $k\neq 0$ and $k\neq p^i$ proving the first part of the result.


For the second part of the result, the only way of partitioning the variables to contribute modulo $p$ is into equal parts,  each part containing one copy of each variable that appears in $m$. Each part is a monomial of $P$ and has coefficient either $0$ or $1$ in $P$ and hence $m$ has coefficient either $0$ or $1$ in $(P)^{p^i}$ correspondingly, but by assumption the coefficient of $m$ in $(P)^{p^i}$ is nonzero.
\end{proof}

The next lemma tells us that how many copies of each variable each factor $(P)^{p^i}$ receives is highly constrained.  This will be useful in the proof of Proposition~\ref{prop main prop} because it implies a kind of independence between the variable assignments to factors for different powers of $i$.

\begin{lemma}\label{lemma prop lem 2}
Let $p$, $q$, $P$, and $Q$ be as in Proposition~\ref{prop main prop}. Let $m = x_1^{k^{(1)}}\cdots x_N^{k^{(N)}}$ be a monomial in $P^{q-1}$ (respectively $Q^{q-1}$).  Let $k^{(j)}_i$ be the $i$-th digit in the base $p$ expansion of $k^{(j)}$.  Let $m_i = x_1^{k^{(1)}_i} \cdots x_N^{k^{(N)}_i}$.  Then
\[
[m]P^{q-1} \equiv \prod_{i=0}^{s-1} [m_i]((P)^{p^i})^{p-1} \bmod p
\]
(respectively with $Q$ in place of $P$).
\end{lemma}

\begin{proof}
Recall \eqref{eq factor P}:
\[
P^{q-1} = \big((P)^{p^{s-1}} \ldots P^{p} P\big)^{p-1}. 
\]
{}From this we can immediately write $[m]P^{q-1}$ as a sum of products of $[m_i']((P)^{p^i})^{p-1}$ where the sum runs over all ways to partition $m$ into monomials $m_i'$.  The only thing to prove, then, is that modulo $p$ the only partition that contributes is the one given by the $m_i$'s defined by the base $p$ expansions.

Now consider a monomial $m_i'$ that appears via the term $[m_i']((P)^{p^i})^{p-1}$ and how it can be partitioned among each of the $(p-1)$ factors of the form $(P)^{p^i}$.  Call each factor $(P)^{p^i}$ a \textit{block} of size $p^i$.
By Lemma~\ref{lemma prop lem 1}, in order for $m_i'$ to contribute modulo $p$, in each block of size $p^i$ each variable either appears $p^i$ times or not at all. Let $\ell^{(j)}_i \in [0,p-1]$ be the number of these blocks in which variable $x_j$ appears. Then the blocks of $P$ of size $p^i$ are assigned $p^i \ell^{(j)}_i$ copies of $x_j$ such that $k^{(j)} = \sum_{i=0}^{s-1} p^i\ell^{(j)}_i$. Therefore the $\ell^{(j)}_i$'s are exactly the unique base $p$ expansion of $k^{(j)}$: $\ell^{(j)}_i = k^{(j)}_i$.
Thus the only term contributing modulo $p$ is the term given by the base $p$ expansions, giving the statement of the lemma.
\end{proof}

\begin{proof}[Proof of Proposition~\ref{prop main prop}]
If the total degree of $PQ$ is not $N$ then the statement holds with $0$ on each side.  Assume now that the total degree of $PQ$ is $N$.

Since $(PQ)^{p-1}$ is of total degree $N(p-1)$ and $P$ and $Q$ are both linear in each variable with all coefficients in $\{0,1\}$, the coefficient of $(x_1\cdots x_N)^{p-1}$ in $(PQ)^{p-1}$ is the number of ways to partition $p-1$ copies of each variable into $2p-2$ ordered parts such that the first $p-1$ parts each give a monomial of $P$ and the second $p-1$ parts each give a monomial of $Q$.   We will return to this fact later in the proof.

Now consider $q$.  As in the proof of the previous lemma, using \eqref{eq factor P} and the analogous equation for $Q$, we can write $(PQ)^{q-1}$ as the product of factors, each with an exponent that is a power of $p$:
\[
P^{q-1}  Q^{q-1} = \big((P)^{p^{s-1}} \ldots P^{p} P\big)^{p-1} \big((Q)^{p^{s-1}} \ldots Q^{p} Q\big)^{p-1}
\]
Using this factorization, the coefficient of $(x_1\cdots x_N)^{q-1}$ in $(PQ)^{q-1}$ can be computed by partitioning $q-1$ copies of each variable into $2s$ ordered parts such that the the $i$-th part is a monomial in $(P)^{p^{s-i}}$ for $1\leq i\leq s$ and is a monomial in $(Q)^{p^{s-(i-s)}}$ for $s < i \leq 2s$, and then taking the product of the coefficients of these monomials in their corresponding power of $P$ or $Q$.

As in the proof of the previous lemma we refer to each factor of the form $(Q)^{p^i}$ or $(P)^{p^i}$ in the expression above as a \textit{block} of size $p^i$. For a given partition of $q-1$ copies of the variables, we say the variables for the part corresponding to this block are \textit{assigned} to the block.  Since we are taking a product of the coefficients we only need to consider assignments for which the coefficients of the assigned monomials are nonzero modulo $p$, and so from now on we restrict ourselves to such assignments.

As a consequence, since $P$ and $Q$ are linear in each variable, a block of size $p^i$ can be assigned between 0 and $p^i$ copies of any variable, and then by Lemma~\ref{lemma prop lem 1} each block of size $p^i$ can be assigned either exactly $0$ or exactly $p^i$ copies of each variable.  

Suppose we assign exactly $k^{(j)}$ copies of $x_j$ to $P^{q-1}$.  Then $q-1-k^{(j)}$ copies of $x_j$ are assigned to $Q^{q-1}$.   By Lemma~\ref{lemma prop lem 2} the variables must be assigned among the blocks of $P$ and the blocks of $Q$ by the base $p$ expansions of the $k^{(j)}$ and the $q-1-k^{(j)}$ respectively.

For each $j$, from $k^{(j)} + (q-1-k^{(j)}) = q-1 = (p-1)(p^{s-1} + \ldots + p + 1)$, we see that the base $p$ expansions of $k^{(j)}$ and $(q-1-k^{(j)})$ sum to the base $p$ expansion $\lambda \ldots \lambda$ where $\lambda = p-1$, and it follows that there is no carry-over when adding the $i$-th base $p$ digits of $k^{(j)}$ and $q-1-k^{(j)}$.  This can be seen inductively: in adding the $0$-th digits in order to obtain $\lambda$ there can be no carry, and continue likewise.  Consequently, the $i$-th digit of $q-1-k^{(j)}$ is $\lambda - k^{(j)}_i$. Because assignments on blocks of size $p^i$ are determined by the $i$-th base $p$ digit of $k^{(j)}$, it follows that assignments on blocks of size $p^i$ are independent of assignments on blocks of size $p^j$ for $i \neq j$ and there are exactly $p^i (p-1)$ copies of each variable between all $2(p-1)$ blocks of size $p^i$.

For each $i$, by Lemma~\ref{lemma prop lem 1}, since we can treat the $p^i$ copies of a variable as a unit when assigning to blocks of size $p^i$, counting the number of ways to assign each variable to the blocks of size $p^i$ is simply counting the number of ways to choose which of the $2(p-1)$ blocks get each variable where each variable is assigned to $p-1$ blocks.  This is the same as the number of ways of assigning $p-1$ copies of each variable between the $2(p-1)$ parts in $(PQ)^{p-1}$  --- exactly as discussed at the beginning of this proof in the second paragraph.   By Lemma~\ref{lemma prop lem 2} and the discussion of the previous paragraph, this assignment can be done independently for all $s$ pairs $((P)^{p^i} (Q)^{p^i})^{p-1}$, and the coefficients match modulo $p$ by the last part of Lemma~\ref{lemma prop lem 1} so it follows that
\[
   [(x_1\cdots x_n)^{q-1}](PQ)^{q-1} = ([(x_1\cdots x_n)^{p-1}](PQ)^{p-1})^s \mod p
\]
as desired.
\end{proof}

\begin{proof}[Proof of Theorem~\ref{thm1}] 
As described in the previous section, since $G$ is the decompletion of a 4-regular graph, $|E(G)|$ is even and there exists a vertex of degree 3, say $v$ with incident edges $1,2,3$ and neighbouring vertices $a$, $b$, $c$ as in Figure~\ref{fig eg G}. Then by Lemmas~\ref{lma1}, \ref{lem 3-inv}, \ref{lma2}, \ref{lma3} as described in Section~\ref{sec edge part} we obtain
\[
    c_2^{(q')}(G)  = -[(\alpha_{4}\cdots \alpha_{N})^{q'-1}](\Phi^{b,ac}_{H} \Psi_H)^{q'-1} \bmod p
\]
where $H=G-v$ for both $q' = p, q$ and where $|E(G)|=N$.  $\Phi^{b,ac}_H$ and  $\Psi_H$  are both linear in each variable and have all coefficients $0$ or $1$ by construction (see Remark~\ref{rem linear}) so Proposition~\ref{prop main prop} and Equation~\ref{eq:c2q_finalform} give 
\begin{align*}
    c_2^{(q)}(G) & =  -[(\alpha_{4}\cdots \alpha_{N})^{q-1}](\Phi^{b,ac}_{H} \Psi_H)^{q-1} \bmod p \\
    & = -([(\alpha_{4}\cdots \alpha_{N})^{p-1}](\Phi^{b,ac}_{H} \Psi_H)^{p-1})^s \bmod p \\
    & = -(-c_2^{(p)}(G))^s \bmod p\\
    & = (-1)^{s+1}(c_2^{(p)}(G))^s \bmod p
\end{align*}
proving the Theorem.
\end{proof}

In particular, this implies that $c^{(q)}_2(G) \equiv c^{(p)}_2(G) \bmod p$ for $p = 2$ and for $s = p$ by Fermat's Little Theorem.

By Theorem~\ref{thm1}, we see that the relation conjectured by Oliver Schnetz (Conjecture~\ref{oliver}) holds modulo $p$. It remains as future work to show or disprove that it holds modulo $q$.  Different techniques will need to be used to do so since the use of Lemma~\ref{lma3} fundamentally restricts us to working modulo $p$. 

Note that Proposition~\ref{prop main prop} does not hold for all polynomials $F=PQ$ as if we take $F=(1+x^8)(1+y^8)$ then with $p=3$, $q=9$ we get $[x^8y^8]F^8=64$ and $([x^2y^2]F^2)^2 = 0$ but $64\equiv 1 \mod 3$.


Note that our result shows that the $c_2$ invariant at a fixed $p$ determines the $c_2$ invariant modulo $p$ at all powers of that same prime $p$.  Moving from working modulo $p$ to working modulo the prime power, the conjecture does not claim that the $c_2$ invariant at a fixed $p$ should determine the $c_2$ invariant at all powers of that prime (modulo the prime power), but rather it claims that knowing the $c_2$ invariant at \textit{all} primes should determine it at all prime powers (modulo the prime power).  A stronger prime-by-prime version of the conjecture would be that knowing the $c_2$ invariant of a graph at a fixed $p$ should determine the $c_2$ invariant at all powers of that prime $p$ (modulo the prime power).  This stronger prime-by-prime version of the conjecture is false as the graphs $P_{8,36}$ and $P_{8,37}$ from \cite{Sphi4} have $c_2^{(2)}=1$ (with different values of $c_2$ at other primes) but $c_2^{(8)}=1$ and $c_2^{(8)}=-1$ respectively\footnote{Thanks to Oliver Schnetz for this calculation}, which are different modulo $8$ but equal modulo $2$ as required by our theorem.  We still expect the full conjecture to be true but, for example, information at primes other than 2 would be needed to determine the value of the $c_2$ invariant at $8$.
  
In conclusion we have used the fruitful edge partitioning approach to prove that Conjecture~\ref{oliver} for the $c_2$ invariant holds modulo $p$, but other techniques will be needed to prove the full conjecture.
\bibliographystyle{plain}
\bibliography{main}

\end{document}